\documentclass[11pt]{article}

\usepackage{amsmath,amssymb,amsfonts,amsthm}
\usepackage{graphicx}
\usepackage{subcaption}
\usepackage{multirow}
\usepackage{bbm}
\usepackage[authoryear,round]{natbib}
\usepackage[colorlinks,citecolor=blue,urlcolor=blue]{hyperref}
\usepackage[margin=1in]{geometry}

\newcommand{\bfw}{\mbox{\bf w}}

\newcommand{\bfx}{\mbox{\bf x}}
\newcommand{\bfy}{\mbox{\bf y}}

\newcommand{\bbF}{\mbox{$\mathbb{F}$}}

\newcommand\raiseT[2]{%
\setbox0\hbox{$#1{#2}$}\raise\dp0\box0}

\newcommand{\bea}{\begin{eqnarray}}
\newcommand{\eea}{\end{eqnarray}}
\newcommand{\ba}{\begin{eqnarray*}}
\newcommand{\ea}{\end{eqnarray*}}
\newcommand{\be}{\begin{equation}}
\newcommand{\ee}{\end{equation}}
\newcommand{\bi}{\begin{itemize}}
\newcommand{\ei}{\end{itemize}}

\newcommand{\var}{\mbox{\sc var}}

\newcommand{\ind}{\mathbbm{1}}

\theoremstyle{plain}

\newtheorem{theorem}{Theorem}[section]

\theoremstyle{definition}

\title{\textbf{Low-Order Inclusion Probabilities\\
Do Not Determine Asymptotic Normality}}

\author{
Jiahua Chen\\[2mm]
Department of Statistics\\
The University of British Columbia\\
Vancouver, BC, Canada\\
\texttt{jhchen@stat.ubc.ca}
}

\date{}

\begin{document}

\maketitle

\begin{abstract}
Central limit theorems in survey sampling are typically established
under specific sampling designs rather than from generic characteristics
of a design. Motivated by the search for more general conditions,
we investigate whether low-order inclusion probabilities can provide
such a basis. First- and second-order inclusion probabilities determine
the first two moments of sampling estimators such as the
Horvitz--Thompson estimator, but we show that even exact agreement of
these probabilities with those of simple random sampling without
replacement (SRSWOR) does not preserve asymptotic normality.
We construct a sequence of fixed-size sampling
designs having exactly the same first- and second-order inclusion
probabilities as SRSWOR, together with a sequence of finite populations
for which the sample total converges to $N(0,1)$ under SRSWOR but to a
centered standard exponential distribution under the alternative
designs. Thus, the two sample totals have exactly the same mean and
variance at every stage of the sequence, yet have different limiting
distributions. The alternative designs are obtained from balanced
incomplete block designs arising from finite projective planes. Their
cyclic structure allows us to prescribe the distribution of the sample
total under block sampling, while a probabilistic permutation argument,
combined with an exponential inequality for sampling without
replacement, yields finite populations satisfying the SRSWOR central
limit theorem. The result shows that low-order inclusion probabilities
do not control the global dependence structure relevant to asymptotic
normality and therefore cannot, by themselves, provide general
conditions for central limit theorems under complex sampling designs.
\end{abstract}

\medskip
\noindent
\textbf{MSC2020 subject classifications.}
Primary 62D05, 60F05; secondary 05B05.

\medskip
\noindent
\textbf{Keywords and phrases.}
Balanced incomplete block design; finite-population asymptotics;
inclusion probability; sampling design; simple random sampling without
replacement; survey sampling.

\bigskip

\section{Introduction}

The central limit theorem (CLT) is one of the fundamental results in probability
and statistics. Beyond its classical formulation for sums of independent
random variables, it represents a much broader phenomenon: the aggregate
effect of a large number of individually non-dominating contributions often
behaves approximately normally. Together with the delta method, Slutsky's
theorem, and related techniques, central limit theorems provide the theoretical
basis for a large part of frequentist statistical inference; see, for example,
\cite{serfling1980}.

Survey sampling provides a particularly interesting setting for studying this
phenomenon. Let
\[
{\cal P}_N=\{y_1,\ldots,y_N\}
\]
be a finite population, and let a probability sampling plan be a probability
distribution on the subsets of $\{1,\ldots,N\}$. Under simple random sampling
without replacement (SRSWOR) of size $n$, every subset of size $n$ is selected
with probability ${N\choose n}^{-1}$. The sample mean
\[
\bar y_n=\frac{1}{n}\sum_{i\in s}y_i
\]
is unbiased for the population mean
\[
\bar Y_N=\frac{1}{N}\sum_{i=1}^N y_i,
\]
with
\[
\var(\bar y_n)
=
\left(\frac1n-\frac1N\right)V_N^2,
\qquad
V_N^2=\frac{1}{N-1}\sum_{i=1}^N(y_i-\bar Y_N)^2.
\]
For a sequence of finite populations and sample sizes tending to infinity,
suitable conditions on the population values imply a finite-population CLT
for $\bar y_n$; classical results include \cite{erdos1959} and
\cite{hajek1960}.

More general sampling plans are often required in applications. Let
$\delta_i$ indicate whether unit $i$ is included in the sample, and define
the first- and second-order inclusion probabilities by
\[
\pi_i=P(\delta_i=1),
\qquad
\pi_{ij}=P(\delta_i=\delta_j=1).
\]
The population total
\[
Y_N=\sum_{i=1}^N y_i
\]
is commonly estimated by the Horv\'ath--Thompson (HT) estimator
\[
\widehat Y_{\rm HT}
=
\sum_{i=1}^N\frac{\delta_i y_i}{\pi_i}.
\]
The expectation of the HT estimator is determined by the first-order
inclusion probabilities, whereas its variance is determined entirely by the
first- and second-order inclusion probabilities. Consequently, two sampling
plans having identical $\pi_i$ and $\pi_{ij}$ produce HT estimators having
exactly the same expectation and variance for every finite population.

In contrast to these simple moment calculations, establishing asymptotic
normality under general sampling plans is considerably more difficult.
Central limit theorems in survey sampling have historically been developed
largely one sampling design at a time. Classical examples include SRSWOR
\citep{erdos1959,hajek1960} and rejective sampling
\citep{Hajek1964}. More recent results likewise typically exploit particular
structures of the sampling design. For example, \citet{berger1998normal}
obtains normal approximations by relating unequal-probability designs to
rejective sampling, while \citet{loonis2019determinantal} exploit the
dependence structure of determinantal sampling designs. Central limit
theorems have also been established for particular multistage settings,
including two-phase sampling \citep{chenrao2007} and two-stage sampling
\citep{chauvet2020}. In methodological developments for more general complex
designs, asymptotic normality of an appropriate HT estimator may instead be
imposed as an additional regularity condition; see, for example,
\citet{boistard2017functional} and \citet{berger2016empirical}.

This design-by-design nature of the existing theory naturally raises the
question of whether a more generic route to asymptotic normality is possible.
Ideally, one would like to identify transparent and verifiable characteristics
shared by broad classes of sampling designs that are sufficient for a CLT.
Inclusion probabilities are natural candidates. They are intrinsic
characteristics of a sampling design, can be specified without describing its
complete probability distribution, and already determine the expectation and
variance of the HT estimator through their first two orders. It is therefore
natural to ask whether sufficiently strong conditions on low-order inclusion
probabilities might also control its asymptotic distribution.

This leads to the basic question considered in this paper:
{\bf
How much do low-order inclusion probabilities tell us about the
sampling distribution of an estimator?
}
More specifically, suppose two sequences of sampling plans have identical
first- and second-order inclusion probabilities. If the HT estimator is
asymptotically normal under one design, must it also be asymptotically normal
under the other?

{\bf The answer is no.}

We construct a sequence of fixed-size sampling plans having exactly the same
first- and second-order inclusion probabilities as SRSWOR, together with a
sequence of finite populations for which the sample total is asymptotically
normal under SRSWOR but converges to a centered exponential distribution
under the alternative sampling plans. Thus, the two estimators have exactly
the same expectation and variance at every stage of the sequence, while their
limiting distributions are fundamentally different. In particular, even
exact agreement with SRSWOR in first- and second-order inclusion probabilities
cannot transfer its central limit theorem to another sampling design.

Constructing such an example is not immediate. The alternative sampling plans
are obtained from a particular sequence of balanced incomplete block designs
(BIBDs) arising from finite projective planes. Their combinatorial structure
ensures exact agreement with SRSWOR in first- and second-order inclusion
probabilities. A cyclic representation of these designs then allows us to
prescribe the distribution of the sample total under BIBD sampling. The
remaining difficulty is to construct a finite population having these
prescribed BIBD sample totals while simultaneously satisfying a
finite-population CLT under SRSWOR. We resolve this problem through a
permutation argument: after randomly assigning the prescribed block totals,
an exponential inequality for sampling without replacement shows that there
exists an assignment for which no population unit dominates. The resulting
finite population therefore satisfies a sufficient condition for the SRSWOR
CLT while retaining the prescribed nonnormal distribution under BIBD sampling.

The result demonstrates a fundamental distinction between local and global
information about a sampling design. Low-order inclusion probabilities
determine corresponding low-order moments of the Horv\'ath--Thompson
estimator, but they need not control the global dependence structure of the
sample indicators. It is this dependence structure that can remain visible
in the limiting distribution even when the first two moments agree exactly.
Consequently, general conditions for asymptotic normality cannot be based on
first- and second-order inclusion probabilities alone. Additional information
about the sampling design is essential.

The example naturally raises a broader question. For a fixed integer $r$,
suppose two sequences of sampling plans agree in all inclusion probabilities
through order $r$. Does such finite-order agreement impose sufficiently strong
restrictions to determine their limiting distributions? We conjecture that it
does not: for every fixed $r$, agreement of inclusion probabilities through
order $r$ alone is insufficient to guarantee a common limiting distribution.
The present paper establishes this assertion only for $r=2$.

The remainder of the paper is organized as follows.
Section~\ref{sec:connection} establishes the connection between block designs and SRSWOR,
showing more generally that a block design of strength $t$ reproduces
the SRSWOR inclusion probabilities through order $t$.
Section~\ref{sec:BIBD} constructs the particular projective-plane BIBDs used in our
counterexample and develops their cyclic representation.
Section~\ref{sec:construction} constructs the finite populations and establishes the
different limiting distributions under BIBD sampling and SRSWOR.
Section~\ref{sec:numerical} provides a numerical demonstration, 
and Section~\ref{sec:discussion} concludes with a brief discussion.

\section{Connecting block designs and SRSWOR}
\label{sec:connection}
Given a finite population of size $N$ and a sample size $n$, simple
random sampling without replacement (SRSWOR) assigns equal selection
probability to all subsets $s$ of size $n$ from
\[
[N]=\{1,2,\ldots,N\}.
\]
Its first- and second-order inclusion probabilities are
\[
\pi_i=\frac{n}{N},
\qquad
\pi_{ij}=\frac{n(n-1)}{N(N-1)},
\]
for all $i\in[N]$ and all distinct $i,j\in[N]$.

We first show that these inclusion probabilities can be reproduced
exactly by sampling from a balanced incomplete block design (BIBD),
even though the resulting sampling plan can be very different from
SRSWOR. More generally, a block design of strength $t$ reproduces all
SRSWOR inclusion probabilities through order $t$. This observation
provides both the starting point for our counterexample and the
motivation for the broader question concerning finite-order inclusion
probabilities raised in the Introduction.

Our construction uses balanced incomplete block designs from
experimental design \citep{beth1999design}. In experimental design,
one may wish to compare the effects of $N$ treatments, while only a
subset of $n<N$ treatments can be applied within each block. A BIBD
consists of $B$ blocks of size $n$ arranged so that every treatment
appears in the same number $R$ of blocks and every pair of distinct
treatments appears together in the same number $\lambda$ of blocks.

Regard the $N$ treatments as the $N$ units of a finite population and
the $B$ blocks as subsets of this population. A sampling plan is
obtained by selecting one of the $B$ blocks uniformly at random and
taking all units in the selected block as the sample. Its first- and
second-order inclusion probabilities are
\[
\pi_i=\frac{R}{B},
\qquad
\pi_{ij}=\frac{\lambda}{B},\quad i\ne j.
\]
The standard BIBD identities
\[
NR=Bn,
\qquad
R(n-1)=\lambda(N-1)
\]
then give
\[
\pi_i=\frac{n}{N},
\qquad
\pi_{ij}=\frac{n(n-1)}{N(N-1)}.
\]
Thus, every BIBD induces a fixed-size sampling plan having exactly the
same first- and second-order inclusion probabilities as SRSWOR of the
same size.
Since both designs have the same constant first-order inclusion
probability $\pi_i=n/N$, equality of the first two moments of the
HT estimators is equivalent to equality of the first two moments
of the corresponding sample totals.

The connection extends naturally to higher-order inclusion
probabilities. Recall that a $t$-$(N,n,\lambda_t)$ design is a
collection of blocks of size $n$ from $N$ treatments such that every
set of $t$ distinct treatments is contained in exactly $\lambda_t$
blocks.

\begin{theorem}
\label{thm:tdesign}
Let ${\cal B}=\{B_1,\ldots,B_B\}$ be a
$t$-$(N,n,\lambda_t)$ design. Consider the sampling plan that selects
one block uniformly from ${\cal B}$ and takes all units in the selected
block as the sample. Then, for every $s=1,\ldots,t$ and every set of
$s$ distinct units $i_1,\ldots,i_s$,
\[
\pi_{i_1\cdots i_s}
=
P(i_1,\ldots,i_s\in S)
=
\frac{n^{(s)}}{N^{(s)}}
=
\frac{\binom ns}{\binom Ns},
\]
where
\[
a^{(s)}=a(a-1)\cdots(a-s+1)
\]
denotes the falling factorial. Consequently, all inclusion
probabilities through order $t$ are identical to those under SRSWOR
of size $n$ from a population of size $N$.
\end{theorem}

\begin{proof}
Fix $s\leq t$ and a set $C$ of $s$ distinct population units.
Let $\lambda_s$ denote the number of blocks containing $C$.
The defining property of a $t$-design implies that $\lambda_s$ does
not depend on the particular choice of $C$.

Count pairs $(C',B_j)$ for which $C'$ is an $s$-subset of the
population and $C'\subset B_j$. Counting first over the $B$ blocks
gives
\[
B\binom ns.
\]
On the other hand, there are $\binom Ns$ possible choices of $C'$,
each contained in exactly $\lambda_s$ blocks. Hence
\[
B\binom ns
=
\lambda_s\binom Ns,
\]
and therefore
\[
\frac{\lambda_s}{B}
=
\frac{\binom ns}{\binom Ns}.
\]
Since a block is selected uniformly,
\[
\pi_{i_1\cdots i_s}
=
\frac{\lambda_s}{B}
=
\frac{\binom ns}{\binom Ns},
\]
which is precisely the corresponding inclusion probability under
SRSWOR.
\end{proof}

Theorem~\ref{thm:tdesign} shows that agreement with SRSWOR in
finite-order inclusion probabilities is not an exceptional property:
it arises naturally from block designs of corresponding strength.
Our counterexample requires only the case $t=2$. What is exceptional,
and considerably more difficult to arrange, is that the resulting
BIBD sampling plan and SRSWOR have different limiting distributions
on the same sequence of finite populations.

We therefore turn to the two ingredients needed for the construction.
First, we identify a particular sequence of BIBDs with enough
algebraic structure to control all possible sample totals under block
sampling. Second, we construct a corresponding sequence of finite
populations for which the SRSWOR central limit theorem nevertheless
holds. For the first task we use BIBDs arising from finite projective
planes. Their construction requires some elementary facts about finite
fields and finite geometry, which we introduce in the next section.

\section{A projective-plane BIBD}
\label{sec:BIBD}

We now construct the particular sequence of BIBDs that will be used in
our counterexample. The construction is based on finite projective planes.
We introduce only the finite-field and finite-geometry facts needed for
our purpose; standard references include \citet{lidl1997finite},
\citet{dembowski1968finite}, and \citet{hughes1973projective}.

\subsection{Finite fields}

In abstract algebra, a field is a set equipped with two operations,
addition and multiplication, satisfying the usual field axioms familiar
from the real numbers $\mathbb{R}$. A finite field is a field containing
finitely many elements. A finite field with $q$ elements exists if and
only if
\[
q=p^m
\]
for some prime number $p$ and positive integer $m$. Up to isomorphism,
the finite field of order $q$ is unique and is denoted by $\bbF_q$.

When $q=p^m$, the finite field $\bbF_{p^m}$ is an $m$-dimensional
vector space over $\bbF_p$. Hence, as vector spaces,
\[
\bbF_{p^m}\cong\bbF_p^m
=
\{(a_0,a_1,\ldots,a_{m-1}):a_i\in\bbF_p\}.
\]
Here
\[
\bbF_p=\{0,1,\ldots,p-1\},
\]
with addition and multiplication performed modulo $p$.

Addition in $\bbF_{p^m}$ corresponds to componentwise addition in
$\bbF_p^m$. To define multiplication, identify each vector with a
polynomial of degree less than $m$:
\[
(a_0,a_1,\ldots,a_{m-1})
\quad\longleftrightarrow\quad
a_0+a_1x+\cdots+a_{m-1}x^{m-1}.
\]
There exists an irreducible polynomial $p^\ast(x)$ of degree $m$ over
$\bbF_p$, which may be chosen to be monic. If $p_1(x)$ and $p_2(x)$
are two polynomial representations, their product in $\bbF_{p^m}$ is
defined by
\[
p_1(x)\cdot p_2(x)
=
p_1(x)p_2(x)\pmod{p^\ast(x)}.
\]
Together with componentwise addition, this multiplication makes the
$p^m$ vectors into a finite field of size $p^m$. See
\citet{lidl1997finite} for further details.

For our construction, we will use the finite field $\bbF_{q^3}$.
It can be regarded as a three-dimensional vector space over $\bbF_q$.
After choosing an $\bbF_q$-basis, every element of $\bbF_{q^3}$ may
therefore be represented by a vector in $\bbF_q^3$.

\subsection{The projective plane $PG(2,q)$}

We introduce the projective geometry $PG(2,q)$ through its
linear-algebra representation.

First, divide the nonzero vectors in $\bbF_q^3$ into equivalence
classes. Two nonzero vectors $\bfx$ and $\bfy$ are equivalent if and
only if
\[
\bfx=c\bfy
\]
for some $c\in\bbF_q^\ast$, where $\bbF_q^\ast=\bbF_q\setminus\{0\}$.
Thus, each equivalence class consists of the nonzero vectors in a
one-dimensional subspace of $\bbF_q^3$. Such an equivalence class is
called a \emph{projective point}. We write
\[
P_{\bfx}
=
\{c\bfx:c\in\bbF_q^\ast\}.
\]
There are $q^3-1$ nonzero vectors in $\bbF_q^3$, and each projective
point contains $q-1$ of them. Therefore, the number of projective
points is
\[
N=\frac{q^3-1}{q-1}=q^2+q+1.
\]

Second, each nonzero vector
\[
\bfx=(x_0,x_1,x_2)\in\bbF_q^3
\]
defines a two-dimensional subspace
\[
L_{\bfx}
=
\{\bfy=(y_0,y_1,y_2)\in\bbF_q^3:
x_0y_0+x_1y_1+x_2y_2=0\}.
\]
Such a two-dimensional subspace represents a \emph{projective line}.
Since
\[
L_{c\bfx}=L_{\bfx},\qquad c\in\bbF_q^\ast,
\]
the projective lines, like the projective points, may be indexed by
equivalence classes of nonzero vectors. It follows that there are also
$N=q^2+q+1$ projective lines.

A two-dimensional subspace contains $q^2-1$ nonzero vectors, which
form
\[
\frac{q^2-1}{q-1}=q+1
\]
projective points. Hence every projective line contains exactly $q+1$
projective points.

Finally, consider two distinct projective points $P_{\bfx_1}$ and
$P_{\bfx_2}$. Since they are distinct, $\bfx_1$ and $\bfx_2$ are
linearly independent and therefore generate the unique two-dimensional
subspace
\[
\operatorname{span}(\bfx_1,\bfx_2)
=
\{a_1\bfx_1+a_2\bfx_2:a_1,a_2\in\bbF_q\}.
\]
Every two-dimensional subspace of $\bbF_q^3$ represents a projective
line. Consequently, every pair of distinct projective points lies on
exactly one projective line.

\subsection{The associated BIBD}
\label{subsec:BIBD}

We now form a BIBD from $PG(2,q)$. Regard the
\[
N=q^2+q+1
\]
projective points as the units of a finite population. Corresponding
to each projective line, define a block consisting of all projective
points lying on that line. Since there are $N$ projective lines and
each contains
\[
n=q+1
\]
projective points, this gives $N$ blocks, each of size $n=q+1$.

Every pair of distinct projective points lies on exactly one projective
line. Hence every pair of distinct units appears together in exactly
one block, so that
\[
\lambda=1.
\]
It remains to determine the number of blocks containing each individual
unit. Fix a projective point $P$. There are $N-1$ other projective
points, and each determines together with $P$ a unique projective line.
Each line through $P$ contains $q$ other points. Therefore, the number
of lines through $P$ is
\[
\frac{N-1}{q}
=
\frac{q^2+q}{q}
=
q+1.
\]
Thus every unit appears in exactly
\[
R=q+1
\]
blocks.

Consequently, $PG(2,q)$ gives a BIBD with parameters
\[
N=q^2+q+1,\qquad
B=N,\qquad
n=q+1,\qquad
R=q+1,\qquad
\lambda=1.
\]
By the result of the preceding section, selecting one of these $N$
blocks uniformly at random defines a fixed-size sampling plan having
exactly the same first- and second-order inclusion probabilities as
SRSWOR of size $n=q+1$ from a population of size $N=q^2+q+1$.

Taking $q=p^m$ for a fixed prime $p$ and letting $m\to\infty$ therefore
gives the required sequence of BIBD sampling plans, with both
$n\to\infty$ and $N\to\infty$.

\subsection{Cyclic representation}
\label{subsec:cyclic}

The cyclic representation of this BIBD plays an important role in both
our theoretical construction and its numerical implementation.
We develop this
representation while the finite-field construction is still at hand.

Recall that $\bbF_{q^3}$ is a three-dimensional vector space over
$\bbF_q$. Let $\bbF_{q^3}^\ast$ denote its multiplicative group of
nonzero elements. Since $\bbF_q^\ast$ is a subgroup of
$\bbF_{q^3}^\ast$, the quotient group
\[
\bbF_{q^3}^\ast/\bbF_q^\ast
\]
has order
\[
N=\frac{q^3-1}{q-1}=q^2+q+1.
\]
Its elements are the cosets
\[
x\bbF_q^\ast
=
\{cx:c\in\bbF_q^\ast\},
\qquad x\in\bbF_{q^3}^\ast.
\]
Under the identification of $\bbF_{q^3}$ with $\bbF_q^3$, these
cosets are precisely the projective points introduced above.

The multiplicative group $\bbF_{q^3}^\ast$ is cyclic
\citep{lidl1997finite}. Thus, there exists a primitive element
$\theta\in\bbF_{q^3}^\ast$ such that
\[
\bbF_{q^3}^\ast
=
\{1,\theta,\theta^2,\ldots,\theta^{q^3-2}\}.
\]
Consequently, the projective points may be ordered as
\[
\bbF_q^\ast,\quad
\theta\bbF_q^\ast,\quad
\theta^2\bbF_q^\ast,\quad
\ldots,\quad
\theta^{N-1}\bbF_q^\ast.
\]
We may therefore label them by
\[
\mathbb Z_N=\{0,1,\ldots,N-1\},
\]
where label $j$ corresponds to the projective point
$\theta^j\bbF_q^\ast$. Under this labeling, multiplication by
$\theta$ sends
\[
j\longmapsto j+1\pmod N.
\]
Thus, multiplication by $\theta$ becomes ordinary cyclic translation
of the point labels.

This cyclic action also acts on the projective lines. Multiplication
by $\theta$ is an invertible $\bbF_q$-linear transformation of
$\bbF_{q^3}$ and therefore maps projective lines to projective lines.
Starting from any projective line, repeated multiplication by $\theta$
generates its cyclic translates. 
By the classical construction of \citet{singer1938theorem}, this action is
transitive on the projective lines.
Hence, if
\[
D\subset\mathbb Z_N
\]
denotes the labels of the $q+1$ projective points lying on one fixed
projective line, then all projective lines are represented by
\[
D+j=\{d+j\pmod N:d\in D\},
\qquad j=0,1,\ldots,N-1.
\]
These $N$ translates are precisely the $N$ blocks of the
projective-plane BIBD.

The set $D$ is commonly called a \emph{Singer difference set}.
Its defining combinatorial property can also be seen directly from the
BIBD. Since every pair of distinct projective points lies on exactly
one projective line, for every nonzero $h\in\mathbb Z_N$ there is
exactly one ordered pair $(d_1,d_2)\in D\times D$ such that
\[
d_1-d_2=h\pmod N.
\]
Thus $D$ is a cyclic $(N,q+1,1)$ difference set.

The cyclic representation makes the sampling design completely
explicit. Once the $q+1$ labels in one set $D$ are known, all $N$
possible samples are obtained by cyclic translation. Sampling from
the BIBD can therefore be implemented simply by selecting
$j$ uniformly from $\mathbb Z_N$ and taking
\[
D+j
\]
as the sample, without explicitly constructing the entire incidence
matrix.

We have now completed the design component of the counterexample.
The next task is quite different: we must construct, along this
sequence of BIBDs, a sequence of finite populations for which the
BIBD sample total has a prescribed nonnormal limiting distribution
while the sample total under the matching SRSWOR design is
asymptotically normal.

\section{Construction of the finite population and limiting distributions}
\label{sec:construction}

We now construct a sequence of finite populations on which the two
sampling plans developed above have different limiting distributions.
Throughout this section,
\[
N=q^2+q+1,
\qquad
n=q+1,
\]
and the BIBD is the projective-plane design constructed in
Section~\ref{subsec:BIBD}. For concreteness, one may take
$q=2^m$ and let $m\to\infty$.

The idea of the construction is to work backward. Rather than first
specifying a finite population and then determining the distribution
of its BIBD sample total, we first prescribe all possible BIBD sample
totals. We then recover a finite population having exactly those
sample totals. The remaining difficulty is to arrange the prescribed
totals so that the resulting finite population also satisfies a
finite-population central limit theorem under SRSWOR.

\subsection{Prescribing the BIBD sample totals}

Let $\bfy=(y_1,\ldots,y_N)^\top$
denote the finite-population response vector, and let $A=[a_{ji}]$ be
the $N\times N$ incidence matrix of the BIBD, where
\[
a_{ji}
=
\ind(\text{unit $i$ belongs to block $j$}).
\]
Thus, rows of $A$ correspond to blocks and columns correspond to
population units. If block $j$ is selected, its sample total is
$\sum_{i=1}^N a_{ji}y_i.$
Consequently, $A\bfy$
is the vector containing all $N$ possible sample totals under the
BIBD sampling plan.

For the projective-plane BIBD, every block contains $q+1$ units and
every pair of distinct blocks has exactly one unit in common. Hence
\[
AA^\top=qI+J.
\]
By duality of the projective plane, the same relation holds for
columns:
\[
A^\top A=qI+J,
\]
where $I$ is the $N\times N$ identity matrix and $J$ is the
$N\times N$ matrix of ones.

We will work with centered vectors. Suppose
\[
\mathbf 1^\top\bfw=0
\]
and define
\[
\bfy=\frac1q A^\top\bfw.
\]
Since $J\bfw=0$,
\[
A\bfy
=
\frac1q AA^\top\bfw
=
\frac1q(qI+J)\bfw
=
\bfw.
\]
Thus, any centered vector $\bfw$ may be prescribed as the complete
vector of possible BIBD sample totals.

We choose $\bfw$ to approximate a centered standard exponential
distribution. Specifically, for $i=1,\ldots,N$, let
\[
w_i
=
-\log\left(1-\frac{i-0.5}{N}\right)
+
\frac1N\sum_{j=1}^N
\log\left(1-\frac{j-0.5}{N}\right).
\]
By construction,
\[
\sum_{i=1}^N w_i=0.
\]
The empirical distribution of $\{w_1,\ldots,w_N\}$ converges to the
distribution of
\[
E-1,
\qquad E\sim\operatorname{Exp}(1).
\]
Moreover,
\[
\frac1N\|\bfw\|^2\longrightarrow 1
\]
and
\[
\max_{1\leq i\leq N}|w_i|=O(\log N).
\]

With
\[
\bfy=\frac1q A^\top\bfw,
\]
the possible BIBD sample totals are exactly the coordinates of
$\bfw$. Since the BIBD sampling plan selects each of its $N$ blocks
with probability $1/N$, its sample total has the empirical
distribution of $\bfw$. It therefore converges in distribution to
the centered exponential law $E-1$.

The population itself is centered. Indeed,
\[
\mathbf 1^\top\bfy
=
\frac1q\mathbf 1^\top A^\top\bfw
=
\frac{q+1}{q}\mathbf 1^\top\bfw
=
0.
\]
In addition,
\[
\|\bfy\|^2
=
\frac1{q^2}\bfw^\top AA^\top\bfw
=
\frac1q\|\bfw\|^2.
\]
Since $N/q^2\to1$ and $N^{-1}\|\bfw\|^2\to1$, it follows that
\[
\|\bfy\|^2\sim q.
\]
Thus, the finite-population variance of the $y_i$'s is of order
$q/N$, or equivalently of order $q^{-1}$.

\subsection{The SRSWOR requirement}

It remains to show that the same finite population can be chosen so
that its sample total under SRSWOR is asymptotically normal.

For SRSWOR of size $n=q+1$, a standard sufficient condition for the
finite-population CLT is
\[
\frac{\max_{i\in[N]}(y_i-\bar Y_N)^2}
     {n^{-1}\sum_{i=1}^N(y_i-\bar Y_N)^2}
\longrightarrow0;
\]
see, for example, \citet{hajek1960} and \citet{liDing2017}.
Under our construction, $\bar Y_N=0$,
\[
\sum_{i=1}^N y_i^2\asymp q,
\qquad
n\asymp q,
\]
so the denominator above is of constant order. It is therefore
sufficient to establish
\[
\max_{i\in[N]}|y_i|=o(1).
\]

There is no obvious reason for this condition to hold for an arbitrary
assignment of the coordinates of $\bfw$ to the blocks. The key
observation is that we are free to permute these prescribed block
totals. Such a permutation leaves their empirical distribution, and
hence the distribution of the BIBD sample total, completely unchanged.
At the same time, it changes the resulting finite population
\[
\bfy=\frac1qA^\top\bfw.
\]
We exploit this freedom to find an assignment for which no population
unit is dominant.

Let $\tau$ be a permutation of $[N]$ and write
\[
\tau(\bfw)
=
(w_{\tau(1)},\ldots,w_{\tau(N)})^\top.
\]
Define
\[
\bfy_\tau
=
\frac1qA^\top\tau(\bfw).
\]
Then
\[
A\bfy_\tau=\tau(\bfw),
\]
so the BIBD sample-total distribution remains the same for every
$\tau$. In contrast, the coordinates of $\bfy_\tau$ depend strongly
on $\tau$.

For a fixed population unit $i$, the $i$th column of $A$ contains
exactly $q+1$ ones. Hence
\[
y_{i\tau} = \frac1q\sum_{j:\,a_{ji}=1}w_{\tau(j)}.
\]
If $\tau$ is chosen uniformly from the $N!$ possible permutations,
the $q+1$ values appearing in this sum form a simple random sample
without replacement of size $q+1$ from the finite population
$\{w_1,\ldots,w_N\}$. Thus,
\[
y_{i\tau}
=
\frac{q+1}{q}\,\overline W_i,
\]
where $\overline W_i$ has the distribution of the mean of an SRSWOR
sample of size $q+1$ from the centered finite population $\bfw$.

There is another useful way to view this observation. Consider all
possible permutations and all population units simultaneously, and write
\[
{\cal Y}
=
\{y_{i\tau}:i\in[N],\ \tau\text{ is a permutation}\},
\]
where elements are counted with their multiplicities. Apart from the
factor $(q+1)/q$, the uniform distribution on ${\cal Y}$ is exactly
the distribution of the sample mean from an SRSWOR sample of size
$q+1$ from $\bfw$. Indeed, as $\tau$ ranges uniformly over all
permutations, every subset of $q+1$ coordinates of $\bfw$ occurs
equally often in the positions corresponding to the $q+1$ blocks
containing a fixed unit $i$.

Since
\[
\frac1N\sum_{j=1}^Nw_j^2\longrightarrow1,
\qquad
\max_j|w_j|=O(\log N),
\]
the SRSWOR sample mean from $\bfw$, after multiplication by
$\sqrt q$, is asymptotically normal.
Thus, a uniformly selected value from ${\cal Y}$ is typically of
order $q^{-1/2}$.
This observation strongly suggests
that a permutation can be found for which all $N$ coordinates
$y_{i\tau}$ are small simultaneously. Asymptotic normality by itself,
however, does not provide the uniform control required for the maximum.
For that purpose we use an exponential inequality.

\subsection{Existence of a suitable permutation}

Recall that $\max_{j\in[N]}|w_j|=O(\log N)$.
For each fixed $i$, $y_{i\tau}$ is, apart from the factor
$(q+1)/q$, the mean of $q+1$ values sampled without replacement from
the centered finite population $\bfw$. The Hoeffding inequality for
sampling without replacement \citep{hoeffding1963}, in the
finite-population form of
\citet[Corollary~1.1, p.~40]{Serfling1974}, therefore implies that,
for some constant $c>0$ independent of $i$ and $q$,
\[
P_\tau\{|y_{i\tau}|>t\}
\leq
2\exp\left\{
-\frac{cqt^2}{(\log N)^2}
\right\}.
\]
Applying the union bound over all $N$ population units gives
\[
P_\tau\left\{
\max_{i\in[N]}|y_{i\tau}|>t
\right\}
\leq
2N\exp\left\{
-\frac{cqt^2}{(\log N)^2}
\right\}.
\]

We now make the choice explicit. Let
\[
t=
\sqrt{\frac{3}{c}}\,
\frac{(\log N)^{3/2}}{\sqrt q}.
\]
Then
\[
\frac{cqt^2}{(\log N)^2}
=
3\log N,
\]
and hence
\[
P_\tau\left\{
\max_{i\in[N]}|y_{i\tau}|>t
\right\}
\leq
2N e^{-3\log N}
=
\frac{2}{N^2}
\longrightarrow0.
\]
Consequently, with probability tending to one over the uniformly
chosen permutation,
\[
\max_{i\in[N]}|y_{i\tau}|
\leq
\sqrt{\frac{3}{c}}\,
\frac{(\log N)^{3/2}}{\sqrt q}.
\]
Since
\[
N=q^2+q+1,
\]
we have $\log N=O(\log q)$ and therefore
\[
\frac{(\log N)^{3/2}}{\sqrt q}\longrightarrow0.
\]
It follows that, for every sufficiently large $q$, there exists at
least one permutation $\tau_q$ such that
\[
\max_{i\in[N]}|y_{i\tau_q}|=o(1).
\]

Fix one such permutation for each $q$ and define
\[
\bfy_q=\frac1qA^\top\tau_q(\bfw).
\]
The resulting sequence of finite populations therefore satisfies the
sufficient condition for the SRSWOR central limit theorem, while the
empirical distribution of the BIBD sample totals remains unchanged.

\subsection{Different limiting distributions}

We are now ready to summarize the construction and its main
consequence.

\begin{theorem}
\label{thm:counterexample}
There exists a sequence of finite populations ${\cal P}_N$ with
\[
N=q^2+q+1,\qquad n=q+1,\qquad q\to\infty,
\]
and, on each population, a projective-plane BIBD sampling plan and
SRSWOR of size $n$, such that the following properties hold.

\begin{enumerate}
\item
The two sampling plans have exactly the same first- and second-order
inclusion probabilities:
\[
\pi_i=\frac{n}{N},
\qquad
\pi_{ij}=\frac{n(n-1)}{N(N-1)},\qquad i\ne j.
\]

\item
Under the BIBD sampling plan, the sample total $T_{\rm BIBD}$
converges in distribution to a centered standard exponential random
variable:
\[
T_{\rm BIBD}\xrightarrow{d}E-1,
\qquad E\sim\operatorname{Exp}(1).
\]

\item
Under SRSWOR, the sample total $T_{\rm SRS}$ is asymptotically
standard normal:
\[
T_{\rm SRS}\xrightarrow{d}N(0,1).
\]
\end{enumerate}

In particular, the two sample totals have exactly the same expectation
and variance for every $q$, while their limiting distributions are
different.
\end{theorem}

\begin{proof}
For each $q$, let $\bfw$ be the centered exponential-quantile vector
defined above and choose a permutation $\tau_q$ satisfying
\[
\max_{i\in[N]}|y_{i\tau_q}|=o(1).
\]
Define the finite population by
\[
\bfy_q=\frac1qA^\top\tau_q(\bfw).
\]

By construction,
\[
A\bfy_q=\tau_q(\bfw).
\]
Since the BIBD sampling plan selects each of its $N$ blocks with
probability $1/N$, the distribution of its sample total is the
empirical distribution of the coordinates of $\tau_q(\bfw)$, which
is identical to that of $\bfw$. The empirical distribution of
$\bfw$ converges to the distribution of $E-1$, where
$E\sim\operatorname{Exp}(1)$. Hence
\[
T_{\rm BIBD}\xrightarrow{d}E-1.
\]

The population is centered, and
\[
\|\bfy_q\|^2
=
\frac1q\|\bfw\|^2.
\]
Since
\[
\frac1N\|\bfw\|^2\longrightarrow1,
\]
we have
\[
\frac1N\|\bfy_q\|^2
=
\frac{1}{qN}\|\bfw\|^2
\sim\frac1q.
\]
The variance of the SRSWOR sample total is therefore
\[
\begin{split}
\var(T_{\rm SRS})
&=
n^2\left(\frac1n-\frac1N\right)V_N^2\\
&=
n\left(1-\frac nN\right)V_N^2,
\end{split}
\]
where
\[
V_N^2=\frac{1}{N-1}\sum_{i=1}^Ny_i^2.
\]
Using $n=q+1$, $N=q^2+q+1$, and
$N^{-1}\|\bfw\|^2\to1$, it follows that
\[
\var(T_{\rm SRS})\longrightarrow1.
\]

Moreover, the permutation $\tau_q$ was chosen so that
\[
\max_{i\in[N]}|y_{i\tau_q}|=o(1).
\]
Together with the preceding variance calculation, this satisfies the
finite-population Lindeberg condition for SRSWOR. Hence
\[
\frac{T_{\rm SRS}}
{\{\var(T_{\rm SRS})\}^{1/2}}
\xrightarrow{d}N(0,1).
\]
Since $\var(T_{\rm SRS})\to1$, Slutsky's theorem gives
\[
T_{\rm SRS}\xrightarrow{d}N(0,1).
\]

Finally, the two sampling plans have identical first- and second-order
inclusion probabilities by Theorem~\ref{thm:tdesign}. Therefore,
for the same finite population, their sample totals have exactly the
same expectation and variance for every $q$. This completes the
proof.
\end{proof}

Theorem~\ref{thm:counterexample} shows that the failure is stronger
than a difference between two standardized distributions. At every
stage of the sequence, the two sampling plans have identical
first- and second-order inclusion probabilities and hence give sample
totals with exactly the same mean and variance. Their common variance
even converges to one. Nevertheless, one sample total converges to
$N(0,1)$ and the other to a centered exponential distribution.
Thus, the difference in their limiting behavior can only arise from
features of the dependence structure that are not captured by
first- and second-order inclusion probabilities.

\section{A numerical demonstration}
\label{sec:numerical}

We give a brief numerical demonstration of the construction. We take
\[
q=2^{10}=1024,
\qquad
N=q^2+q+1=1,049,601,
\qquad
n=q+1=1025.
\]
Thus, the finite population contains more than one million units,
while the sampling fraction is approximately $10^{-3}$.

To implement the projective-plane BIBD, we use its cyclic
representation from Section~\ref{subsec:cyclic}. A Singer difference
set
\[
D\subset\mathbb Z_N,
\qquad |D|=q+1,
\]
was generated computationally using the Python package
\texttt{galois} \citep{Hostetter_Galois_2020}. Once $D$ is available,
the $N$ blocks of the BIBD are simply
\[
D+j,\qquad j=0,1,\ldots,N-1,
\]
with addition modulo $N$. This is the computational form of the
classical cyclic construction of \citet{singer1938theorem} discussed in
Section~\ref{subsec:cyclic}.
We form the centered exponential-quantile vector $\bfw$ as in
Section~\ref{sec:construction}. Rather than searching for a specially
chosen permutation satisfying the bound used in the proof, we simply
generate one permutation $\tau$ uniformly at random and construct
\[
\bfy=\frac1qA^\top\tau(\bfw).
\]
The theory in Section~\ref{sec:construction} shows that a randomly
chosen permutation has the required property with probability tending
to one. Thus, no optimization over permutations is needed for this
demonstration.

\begin{figure}[t]
\centering
\begin{subfigure}[t]{0.495\textwidth}
    \centering
    \includegraphics[width=\textwidth,height=0.35\textheight,keepaspectratio]{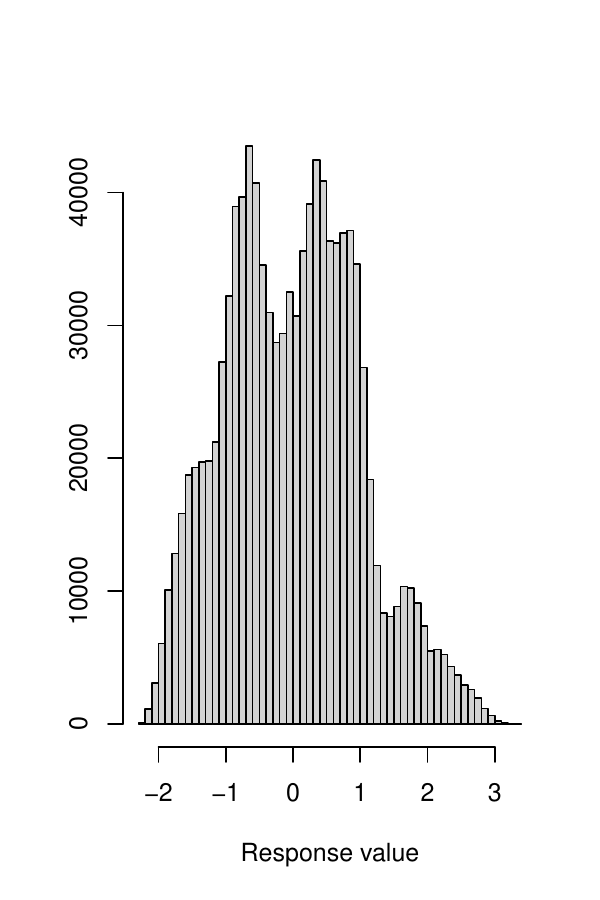}
    \caption{Finite population}
\end{subfigure}
\hfill
\begin{subfigure}[t]{0.495\textwidth}
    \centering
    \includegraphics[width=\textwidth,height=0.35\textheight,
    keepaspectratio]{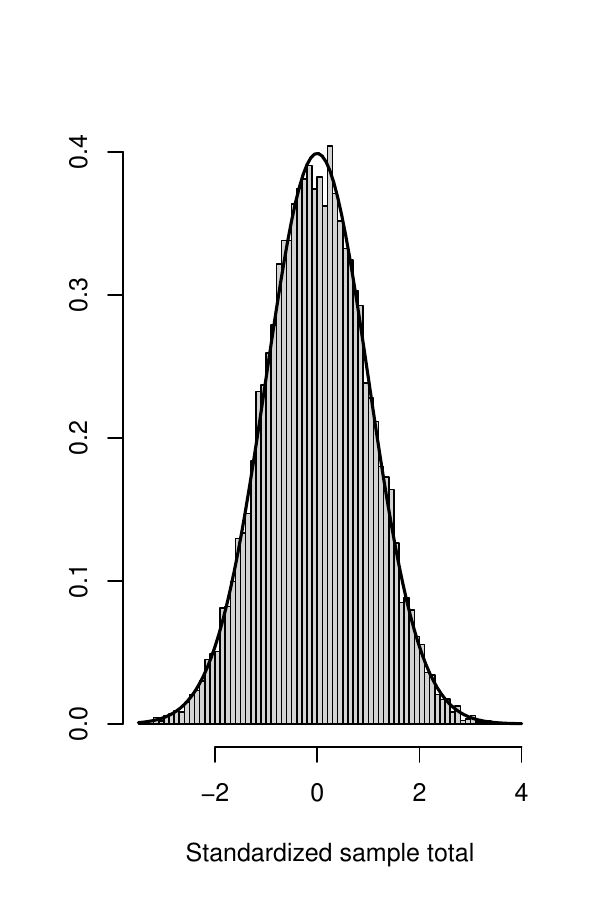}
    \caption{SRSWOR sample totals}
\end{subfigure}
\caption{Histograms of the finite population values and the sample totals
of $20{,}000$ SRSWOR samples.}
\label{fig:demo}
\end{figure}

There is no need to simulate the BIBD sampling distribution. By
construction,
\[
A\bfy=\tau(\bfw),
\]
so its $N$ possible sample totals are known exactly and their empirical
distribution is the centered exponential distribution prescribed by
$\bfw$. The more interesting question is whether the same constructed
finite population behaves normally under SRSWOR. We therefore draw
$20{,}000$ independent SRSWOR samples of size $n=1025$ from $\bfy$.
Figure~\ref{fig:demo} shows, in the left panel, the empirical
distribution of the finite-population values $y_i$ and, in the right
panel, the empirical distribution of the resulting SRSWOR sample
totals. Despite the nonstandard way in which the finite population is
constructed, its SRSWOR sample total is already very close to normally
distributed.

The numerical example is not needed for the theoretical result. Its
purpose is simply to illustrate that the construction is practical:
a single random permutation produces a finite population for which the
SRSWOR central limit behavior is clearly visible, while the BIBD
sample-total distribution remains, by construction, centered
exponential.

\section{Discussion}
\label{sec:discussion}

The counterexample in this paper reveals a limitation of low-order
descriptions of a sampling design that appears not to have been widely
recognized. First- and second-order inclusion probabilities occupy a
central place in survey sampling: they determine the expectation and
variance of the Horv\'ath--Thompson estimator and enter directly into
variance estimation and many asymptotic arguments. Nevertheless, our
result shows that even exact agreement of these quantities with those
of SRSWOR does not preserve asymptotic normality. Two sampling designs
can give estimators with exactly the same mean and variance at every
stage of an asymptotic sequence and yet have fundamentally different
limiting distributions.

The example also helps explain why central limit theorems for complex
sampling designs typically require information beyond low-order
inclusion probabilities. Successful results often exploit additional
structure of the sampling design, impose conditions on higher-order
dependence, or assume an appropriate HT central limit theorem as part
of the asymptotic framework. Our result shows that such additional
information cannot, in general, be replaced by conditions involving
only first- and second-order inclusion probabilities.

More broadly, Theorem~\ref{thm:tdesign} shows that agreement with
SRSWOR through any fixed order $t$ arises naturally from block designs
of strength $t$. This observation suggests that the phenomenon
demonstrated here for $t=2$ may extend much further. We conjecture
that, for every fixed $t$, agreement of all inclusion probabilities
through order $t$ is insufficient by itself to determine the limiting
distribution of a sampling estimator. Establishing or disproving this
conjecture would help clarify what information about a sampling design
is genuinely needed for general finite-population central limit
theorems.

\section*{Acknowledgments}

The author gratefully acknowledges the extensive assistance of ChatGPT
(OpenAI) during the development of this work. ChatGPT was used for
mathematical discussion and exploration, literature searches, checking
and organizing arguments, and improving the presentation of the
manuscript. In particular, interactions with ChatGPT were helpful in
developing and clarifying the BIBD construction and the permutation
argument used in the proof. The author independently verified the
mathematical arguments and references presented in the paper and takes
full responsibility for its content.

\medskip
\noindent
\textbf{Funding.}
The author was supported by the Natural Sciences and Engineering
Research Council of Canada (RGPIN-2025-03989).

\bibliographystyle{plainnat}
\bibliography{chen2023CLT}

\end{document}